\documentclass[10pt]{studiamnew}
\usepackage{graphicx}
\usepackage{amsmath}
\usepackage{amssymb}
\newtheorem{theorem}{Theorem}[section]
\newtheorem{lemma}[theorem]{Lemma}

\newtheorem{definition}[theorem]{Definition}

\theoremstyle{definition}
\newtheorem{remark}[theorem]{Remark}

\renewcommand{\theequation}{\thesection.\arabic{equation}}
\numberwithin{equation}{section}

\begin{document}
\setcounter{page}{1}
\setcounter{firstpage}{1}
\setcounter{lastpage}{4}
\renewcommand{\currentvolume}{??}
\renewcommand{\currentyear}{??}
\renewcommand{\currentissue}{??}
\title[On Hilfer Fractional BVP with nonlocal boundary conditions]{Existence of solution for Hilfer fractional differential problem with nonlocal boundary condition}
\author[H.A. Wahash]{Hanan A. Wahash}
\address{``Department of Mathematics'',\\ Dr.Babasaheb Ambedkar Marathwada University,\\
Aurangabad - 431001, (M.S) India}
\email{hawahash86@gmail.com}
\author[M.S. Abdo]{Mohammed S. Abdo}
\address{``Department of Mathematics'',\\ Dr.Babasaheb Ambedkar Marathwada University,\\
Aurangabad - 431001, (M.S) India}
\email{msabdo1977@gmail.com}
\author[S.K. Panchal]{Satish K. Panchal}
\address{``Department of Mathematics'',\\ Dr.Babasaheb Ambedkar Marathwada University,\\
Aurangabad - 431001, (M.S) India}
\email{drpanchalskk@gmail.com}
\author[S.P. Bhairat$^*$]{Sandeep P. Bhairat\footnote{Corresponding author email: sp.bhairat@marj.ictmumbai.edu.in}}
\address{``Faculty of Engineering Mathematics'',\\ Institute of Chemical Technology Mumbai,\\
Marathwada Campus, Jalna - 431 203 (M.S) India.}
\email{sp.bhairat@marj.ictmumbai.edu.in}
\begin{abstract}
This paper is devoted to study the existence of a solution to Hilfer fractional differential equation with nonlocal boundary condition. We use the equivalent integral equation to study the considered Hilfer differential problem with nonlocal boundary condition. The M\"{o}nch type fixed point theorem and the measure of the noncompactness technique are the main tools in this study. We demonstrate the existence of a solution with a suitable illustrative example.
\end{abstract}
\maketitle

\section{Introduction}
The calculus of arbitrary order has been extensively studied in the last
four decades. It has been proved to be an adequate tool in almost all
branches of science and engineering. Because of its widespread applications,
fractional calculus is becoming an integral part of applied mathematics
research. Indeed, fractional differential equations have been found useful
to describe abundant phenomena in physics and engineering, and the modest
amount of work in this direction has taken place, see \cite{ABLZ,AG,KD} and
references therein. For basic development and theoretical applications of
fractional differential equations, see \cite{HI,KL1}.

In the past two decades, the fractional differential equations are
extensively studied for existence, uniqueness, continuous dependence and
stability of the solution. For some fundamental results in existence theory
of various fractional differential problems with initial and boundary
conditions, see survey papers \cite{ABLZ,AG}, the monograph \cite{KL1}, the
research papers \cite{AP1,AP2,SP1,SPN,SP5,DB1,DBN,SPB,DB2,KD,FK,KM,FMG,HLT,VE,WZ} and references
therein.

In the year 2018, Thabet et al. \cite{SA} investigated the existence of a
solution to BVP for Hilfer FDEs:
\begin{equation}
D_{a^{+}}^{\mu ,\nu }z(t)=f\left( t,z(t),Sz(t)\right) ,0<\mu <1,0\leq \nu
\leq 1,\text{\qquad\ \ }t\in (a,b],  \label{11}
\end{equation}%
\begin{equation}
I_{a^{+}}^{1-\gamma }\left[ uz(a^{+})+vz(b^{-})\right] =w,\text{ }\ \mu \leq
\gamma =\mu +\nu (1-\mu ),u,v,w\in
\mathbb{R}
,  \label{12}
\end{equation}%
by using the M\"{o}nch fixed point theorem.

Recently, in \cite{APB}, Abdo et al. obtained the existence of the solutions
of BVP for the class of Hilfer FDEs:
\begin{equation}
D_{a^{+}}^{\mu ,\nu }z(t)=f(t,z(t)),\text{ \ }p-1<\mu <p,\,0\leq \nu \leq
1\qquad \qquad \ \ \   \label{e8.1a}
\end{equation}%
\begin{equation}
I_{a^{+}}^{1-\gamma }\left[ cz(a^{+})+dz(b^{-})\right] =e,\text{\ \ \ }\mu
\leq \gamma =\mu +\nu (1-\mu ),\qquad \ \ \   \label{e8.1b}
\end{equation}%
by using the Schauder, Schaefer and Krasnosel'skii's fixed point theorems.

Motivated by works cited above, in this paper, we consider the nonlocal
boundary value problem for a class of Hilfer fractional differential
equations (HNBVP):
\begin{equation}
D_{a^{+}}^{\mu ,\nu }z(t)=f(t,z(t)),\text{ \ }0<\mu <1,\,0\leq \nu \leq
1,t\in (a,b],\qquad \ \ \ \ \ \ \ \qquad \qquad  \label{e8.1}
\end{equation}%
\begin{equation}
I_{a^{+}}^{1-\gamma }cz(a^{+})+I_{a^{+}}^{1-\gamma
}dz(b^{-})=\sum_{k=1}^{m}\lambda _{k}z(\tau _{k}),\tau _{k}\in (a,b],\ \mu
\leq \gamma =\mu +\nu -\mu \nu ,  \label{e8.2}
\end{equation}%
where $D_{a^{+}}^{\mu ,\nu }$ is the generalized Hilfer fractional
derivative of order $\mu $ and type $\nu $, $I_{a^{+}}^{1-\gamma }$ is the
Riemann-Liouville fractional integral of order $1-\gamma $, $f:(a,b]\times
\mathbb{R}\rightarrow \mathbb{R}$ be a function such that $f(t,z)\in
C_{1-\gamma }[a,b]$ for any $z\in C_{1-\gamma }[a,b]$ and $c,d\in \mathbb{R}$%
, for $k=1,2,\cdots,m$.\newline
The measure of noncompactness technique and a fixed point theorem of Monch
type are the main tools in this analysis.

The paper is organized as follows: Some preliminary concepts related to our
problem are listed in Section 2 which will be useful in the sequel. In
Section 3, we first establish an equivalent integral equation of BVP and
then we present the existence of its solution. An illustrative example is
provided in the last section.

\section{Preliminaries}

In this section, we present some definitions, lemmas and weighted spaces
which are useful in further development of this paper.

Let $J_{1}=[a,b]$ and $J_{2}=(a,b]\-\infty <a<b<+\infty.$ Let $C(J_{1},E),$ $AC(J_{1},E)$ and $C^{n}(J_{1},E)$ be the Banach spaces of all
continuous, absolutely continuous, $p-$times continuous and continuously
differentiable functions on $J_{1},$ respectively. Here $L^{p}(J_{1},E),$ $%
p>1,$ is the Banach space of measurable functions on $J_{1}$ with the $L^p$
norm where
\begin{equation*}
\left\Vert p\right\Vert _{L^{p}}=\left( \int_{a}^{b}\left\vert
p(s)\right\vert ^{p}ds\right) ^{\frac{1}{p}}<\infty.
\end{equation*}
Let $L^{\infty }(J_{1},E)$ be the Banach space of measurable functions $%
z:J_{1}\longrightarrow E$ which are bounded and equipped with the norm $%
\left\Vert z\right\Vert _{L^{\infty }}=\inf \{e>0:\left\Vert z\right\Vert
\leq e,$ a.e $t\in J_{1}\}.$ Moreover, for a given set $\mathcal{V}$ of
functions $v:J_{1}\longrightarrow E$ let us denote by%
\begin{equation*}
\mathcal{V(}t)=\{v(t):v\in \mathcal{V};t\in J_{1}\},
\end{equation*}%
\begin{equation*}
\mathcal{V(}J_{1})=\{v(t):v\in \mathcal{V};t\in J_{1}\}.
\end{equation*}

\begin{definition}
\cite{KL1} Let $g:[a,\infty )\rightarrow R$ is a real valued continuous
function. The left sided Riemann-Liouville fractional integral of $g$ of
order $\mu >0$ is defined by
\begin{equation}
I_{a^{+}}^{\mu }g(t)=\frac{1}{\Gamma (\mu )}\int_{a}^{t}(t-s)^{\mu
-1}g(s)ds,\quad t>a,  \label{d1}
\end{equation}%
where $\Gamma (\cdot )$ is the Euler's Gamma function and $a\in
\mathbb{R}
.$ provided the right hand side is pointwise defined on $(a,\infty ).$
\end{definition}

\begin{definition}
\cite{KL1} Let $g:[a,\infty )\rightarrow R$ is a real valued continuous
function. The left sided Riemann-Liouville fractional derivative of $g$ of
order $\mu >0$ is defined by
\begin{equation}
D_{a^{+}}^{\mu }g(t)=\frac{1}{\Gamma (p-\mu )}\frac{d^{n}}{dt^{n}}%
\int_{a}^{t}(t-s)^{n-\mu -1}g(s)ds,  \label{d2}
\end{equation}%
where $n=[\mu ]+1,$ and $[\mu ]$ denotes the integer part of $\mu .$
\end{definition}

\begin{definition}
\label{7} \cite{HI} The left sided Hilfer fractional derivative of
function $g\in L^{1}(a,b)$ of order $0<\mu <1$ and type $0\leq \nu \leq 1$
is denoted as $D_{a^{+}}^{\mu ,\nu }$ and defined by
\begin{equation}
D_{a^{+}}^{\mu ,\nu }g(t)=I_{a^{+}}^{\nu (1-\mu )}D^{p}I_{a^{+}}^{(1-\nu
)(1-\mu )}g(t),\text{ }D^{n}=\frac{d^{n}}{dt^{n}}.  \label{d3}
\end{equation}%
where $I_{a^{+}}^{\mu }$ and $D_{a^{+}}^{\mu }$ are Riemann-Liouville
fractional integral and derivative defined by \eqref{d1} and \eqref{d2},
respectively.
\end{definition}

\begin{remark}
\label{rem8.a} From Definition \ref{7}, we observe that:

\begin{itemize}
\item[(i)] The operator $D_{a^{+}}^{\mu ,\nu }$ can be written as
\begin{equation*}
D_{a^{+}}^{\mu ,\nu }=I_{a^{+}}^{\nu (1-\mu )}D^{p}I_{a^{+}}^{(1-\gamma
)}=I_{a^{+}}^{\nu (1-\mu )}D^{\gamma },~~~~~~~~\gamma =\mu +\nu -\mu \nu
\text{.}
\end{equation*}

\item[(ii)] The Hilfer fractional derivative can be regarded as an
interpolator between the Riemann-Liouville derivative ($\nu =0$) and Caputo
derivative ($\nu =1$) as
\begin{equation*}
D_{a^{+}}^{\mu ,\nu }=%
\begin{cases}
DI_{a^{+}}^{(1-\mu )}=~D_{a^{+}}^{\mu },~~~~~~~~~~if~\nu =0; \\
I_{a^{+}}^{(1-\mu )}D=~^{c}D_{a^{+}}^{\mu },~~~~~~~~if~\nu =1.%
\end{cases}%
\end{equation*}

\item[(iii)] In particular, if $\gamma =\mu +\nu -\mu \nu ,$ then
\begin{equation*}
(D_{a^{+}}^{\mu ,\nu }g)(t)=\Big(I_{a^{+}}^{\nu (1-\mu )}\Big(%
D_{a^{+}}^{\gamma }g\Big)\Big)(t),
\end{equation*}%
where $\Big(D_{a^{+}}^{\gamma }g\Big)(t)=\frac{d}{dt}\Big(I_{a^{+}}^{(1-\nu
)(1-\mu )}g\Big)(t).$
\end{itemize}
\end{remark}

\begin{definition}
\cite{KL1} Let $0\leq \gamma <1.$\ The weighted spaces $C_{\gamma }[a,b]$
and $C_{1-\gamma }^{n}[a,b]$ are defined by
\begin{equation*}
C_{\gamma }[a,b]=\{g:(a,b]\rightarrow \mathbb{R}:(t-a)^{\gamma }g(t)\in
C[a,b]\},
\end{equation*}%
and
\begin{equation*}
C_{\gamma }^{n}[a,b]=\{g:(a,b]\rightarrow \mathbb{R},g\in
C^{n-1}[a,b]:g^{(n)}(t)\in C_{\gamma }[a,b]\},\,n\in \mathbb{%
\mathbb{N}
}
\end{equation*}%
with the norms%
\begin{equation*}
{\Vert g\Vert }_{C_{\gamma }}={\Vert }(t-a)^{\gamma }{{g}\Vert }_{C}=\max
\{\left\vert (t-a)^{\gamma }{g(t)}\right\vert :t\in \lbrack a,b]\},
\end{equation*}%
and
\begin{equation}
{\Vert g\Vert }_{C_{1-\gamma }^{n}}=\sum_{k=0}^{n-1}{\Vert g^{(k)}\Vert }%
_{C}+{\Vert g^{(n)}\Vert }_{C_{1-\gamma }},  \label{n1}
\end{equation}%
respectively. Furthermore we recall following weighted spaces
\begin{equation}
C_{1-\gamma }^{\mu ,\nu }[a,b]=\big\{g\in {C_{1-\gamma }[a,b]}%
:D_{a^{+}}^{\mu ,\nu }g\in {C_{1-\gamma }[a,b]}\big\},\quad \gamma =\mu +\nu
(1-\mu )  \label{w1}
\end{equation}%
and%
\begin{equation*}
C_{1-\gamma }^{\gamma }[a,b]=\big\{g\in {C_{1-\gamma }[a,b]}%
:D_{a^{+}}^{\gamma }g\in {C_{1-\gamma }[a,b]}\big\},\quad \gamma =\mu +\nu
(1-\mu ).
\end{equation*}%
where Let $0<\mu <1,0\leq \nu \leq 1$ and $\gamma =\mu +\nu -\mu \nu $.
Clearly, $D_{a^{+}}^{\mu ,\nu }g=I_{a^{+}}^{\nu (1-\mu )}D_{a^{+}}^{\gamma
}g $ and $C_{1-\gamma }^{\gamma }[a,b]\subset C_{1-\gamma }^{\mu ,\nu
}[a,b]. $
\end{definition}

\begin{lemma}
\label{def8.5} \cite{KD} If $\mu >0$ and $\nu >0,$ and $g\in L^{1}(a,b)$
for $t\in \lbrack a,b]$, then the following properties hold:
\begin{equation*}
\Big(I_{a^{+}}^{\mu }I_{a^{+}}^{\nu }g\Big)(t)=\Big(I_{a^{+}}^{\mu +\nu }g %
\Big)(t)\,\, \text{and }\Big(D_{a^{+}}^{\mu }I_{a^{+}}^{\nu }g\Big)(t)=g(t).
\end{equation*}
In particular, if $f\in C_{\gamma }[a,b]$ or $f\in C[a,b]$, then the above
properties hold for each $t\in (a,b]$ or $t\in \lbrack a,b]$ respectively.
\end{lemma}

\begin{lemma}
\label{Le1}\cite{KL1} For $t>a,$ we have

\begin{description}
\item[(i)] $I_{a^{+}}^{\mu }(t-a)^{\delta -1}=\frac{\Gamma (\delta )}{\Gamma
(\delta +\mu )}(t-a)^{\delta +\mu -1},\quad \mu \geq 0,\delta >0.$\newline

\item[(ii)] $D_{a^{+}}^{\mu }(t-a)^{\mu -1}=0,\quad \mu \in (0,1).$
\end{description}
\end{lemma}

\begin{lemma}
\label{def8.8} \cite{HI} Let $\mu >0$, $\nu >0$ and $\gamma =\mu +\nu -\mu
\nu .$ If $g\in C_{1-\gamma }^{\gamma }[a,b]$, then\newline
\begin{equation*}
I_{a^{+}}^{\gamma }D_{a^{+}}^{\gamma }g=I_{a^{+}}^{\mu }D_{a^{+}}^{\mu ,\nu
}g,~D_{a^{+}}^{\gamma }I_{a^{+}}^{\mu }g=D_{a^{+}}^{\nu (1-\mu )}g.
\end{equation*}
\end{lemma}

\begin{lemma}
\label{Le2} \cite{HI} Let $0<\mu <1,$ $0\leq \nu \leq 1$ and $g\in
C_{1-\gamma }[a,b].$ Then%
\begin{equation*}
I_{a^{+}}^{\mu }D_{a^{+}}^{\mu ,\nu }g(t)=g(t)-\frac{I_{a^{+}}^{(1-\nu
)(1-\mu )}g(a)}{\Gamma (\mu +\nu (1-\mu ))}(t-a)^{\mu +\nu (1-\mu )-1},\quad
\text{for all}\quad t\in (a,b],
\end{equation*}%
Moreover, if $\ \gamma =\mu +\nu -\mu \nu ,$ $g\in C_{1-\gamma }[a,b]$ and $%
I_{a^{+}}^{1-\gamma }g\in C_{1-\gamma }^{n}[a,b],$ then
\begin{equation*}
I_{a^{+}}^{\gamma }D_{a^{+}}^{\gamma }g(t)=g(t)-\frac{I_{a^{+}}^{1-\gamma
}g(a)}{\Gamma (\gamma )}(t-a)^{\gamma -1},\quad \text{for all}\quad t\in
(a,b].
\end{equation*}
\end{lemma}

\begin{lemma}
\label{def8.7} \cite{HLT} If $0\leq \gamma <1$ and $g\in C_{\gamma }[a,b]$%
, then
\begin{equation*}
(I_{a^{+}}^{\mu }g)(a)=\lim_{t\rightarrow a^{+}}I_{a^{+}}^{\mu
}g(t)=0,~0<\mu \leq \gamma .
\end{equation*}
\end{lemma}

\begin{lemma}
\cite{MH} Let $E$ be a Banach space and let$\ \Upsilon _{E}$ be the bounded
subsets of $E$. The Kuratowski measure of noncompactness is the map ${\Large %
\alpha }:\Upsilon _{E}\longrightarrow \lbrack 0,\infty )$defined by%
\begin{equation*}
{\Large \alpha }(\mathcal{S})=\inf \{\varepsilon >0:\mathcal{S}\subset \cup
_{i=1}^{m}\mathcal{S}_{i}\text{ and the diam }(\mathcal{S}_{i})\leq
\varepsilon \};\mathcal{S}\subset \Upsilon _{E}.
\end{equation*}
\end{lemma}

\begin{lemma} \cite{GD}
For all nonempty subsets $\mathcal{S}_{1},\mathcal{S}_{2}\subset E$. The
Kuratowski measure of noncompactness ${\Large \alpha }(\mathcal{\cdot })$
satisfies the following properties:
\end{lemma}

\begin{enumerate}
\item ${\Large \alpha }(\mathcal{S})=0\Longleftrightarrow \overline{\mathcal{%
S}}$ is compact ($\mathcal{S}$ is relatively compact);

\item ${\Large \alpha }(\mathcal{S})={\Large \alpha }(\overline{\mathcal{S}}%
)={\Large \alpha }(conv\mathcal{S}),$ where where $\overline{\mathcal{S}}$
and $conv\mathcal{S}$ denote the closure and convex hull of the bounded set $%
\mathcal{S}$ respectively;

\item $\mathcal{S}_{1}\subset \mathcal{S}_{2}\Longrightarrow {\Large \alpha }%
(\mathcal{S}_{1})\leq {\Large \alpha }(\mathcal{S}_{2});$

\item ${\Large \alpha }(\mathcal{S}_{1}+\mathcal{S}_{2})\leq {\Large \alpha }%
(\mathcal{S}_{1})+{\Large \alpha }(\mathcal{S}_{2}),$ where $\mathcal{S}_{1}+%
\mathcal{S}_{2}=\{s_{1}+s_{2}:s\in \mathcal{S}_{1},s\in \mathcal{S}_{2}\};$

\item ${\Large \alpha }(\kappa \mathcal{S})=\left\vert \kappa \right\vert
{\Large \alpha }(\overline{\mathcal{S}}),$ $\kappa \in
\mathbb{R}
;$
\end{enumerate}
For more details, see \cite{APB,SPN,GC}.
\begin{lemma}\cite{MH}
Let $\mathbb{B}$ be a bounded, closed and convex subset of a Banach
space $E$ such that $0\in \mathbb{B}$; and let ${\large \mathcal{T}}$ be a
continuous mapping of $\mathbb{B}$ into itself. If for every subset ${\large
\mathcal{V}}$ of $\mathbb{B}$%
\begin{equation*}
\mathcal{V=}\overline{co}\mathcal{T}(\mathcal{V})\text{ or }\mathcal{V=T}(%
\mathcal{V})\cup \{0\}\Longrightarrow {\Large \alpha }(\mathcal{V}){\large
\mathcal{=}}0
\end{equation*}
holds. Then ${\large \mathcal{T}}$ has a fixed point.
\end{lemma}

\begin{lemma}\cite{SZ}
Let $\mathbb{B}$ be a bounded, closed and convex subset of a Banach
space $C(J_{1},E)$; $F$ is a continuous function on $J_{1}\times J_{1}$; and
a function $f:J_{1}\times E\longrightarrow E$ satisfying the Carath\'{e}%
odory conditions, and assume there exists $\rho \in $ $L^{P}(J_{1},%
\mathbb{R}
^{+})$ such that, for each $t\in J_{1}$ and each bounded set $\mathbb{B}%
^{\ast }\subset E$; one has%
\begin{equation*}
\underset{r\longrightarrow 0^{+}}{\lim }{\Large \alpha }(f(J_{t,r}\times
\mathbb{B}^{\ast }))\leq \rho (t){\Large \alpha }(\mathbb{B}^{\ast }),\text{%
where }J_{t,r}\in \lbrack t-r,t]\cap J_{1}.
\end{equation*}%
If $\mathcal{V}$ is an equicontinuous subset of $\mathbb{B}$; then%
\begin{equation*}
{\Large \alpha }\bigg(\bigg\{\int_{J_{1}}F(t,s)f(s,z(s))ds:z\in \mathcal{V%
\bigg\}\bigg)}\leq \int_{J_{1}}\left\Vert F(t,s)\right\Vert \rho (s){\Large %
\alpha }(\mathcal{V(}s\mathcal{)})ds.
\end{equation*}
\end{lemma}

\begin{lemma}
\label{le}\cite{SPB} Let $\gamma =\mu +\nu -\mu \nu $ where $0<\mu <1$ and
$0\leq \nu \leq 1.$ Let $f:J_{2}\times E\rightarrow E$ be a function such
that $f(t,z)\in C_{1-\gamma }(J_{1},E)$ for any $z\in C_{1-\gamma
}(J_{1},E). $ If $z\in C_{1-\gamma }^{\gamma }(J_{1},E),$ then $z$ satisfies
IVP \eqref{e8.1a}-\eqref{e8.1b} if and only if $z$ satisfies the Volterra
integral equation
\begin{equation}
z(t)=\frac{z_{a}}{\Gamma (\gamma )}(t-a)^{\gamma -1}+\frac{1}{\Gamma (\mu )}%
\int_{a}^{t}(t-s)^{\mu -1}f(s,z(s))ds,\quad t>a.  \label{s3}
\end{equation}
\end{lemma}

\section{Main results}

Now we prove the existence of solution of HNBVP \eqref{e8.1}-\eqref{e8.2} in
$C_{1-\gamma }^{\gamma }(J_{1},E)\subset C_{1-\gamma }^{\mu ,\nu }(J_{1},E).$

\begin{definition}
A function $z\in $ $C_{1-\gamma }^{\gamma }(J_{1},E)$ is said to be a
solution of HNBVP \eqref{e8.1}-\eqref{e8.2} if $z$ satisfies the
differential equation $D_{a^{+}}^{\mu ,\nu }z(t)=f(t,z(t))$ on $(a,b]$, and
the nonlocal condition $\displaystyle I_{a^{+}}^{1-\gamma }\left[
cz(a^{+})+dz(b^{-})\right] =\sum_{k=1}^{m}\lambda _{k}z(\tau _{k}).$
\end{definition}

In the beginning, we need the following axiom lemma:

\begin{lemma}
\label{lee1} Let $0<\mu <1$, $0\leq \nu \leq 1$ where $\gamma =\mu +\nu -\mu
\nu $, and $f:J_{2}\times \mathbb{R}\rightarrow \mathbb{R}$ be a function
such that $f(t,z)\in C_{1-\gamma }(J_{1},E)$ for any $z\in C_{1-\gamma
}(J_{1},E).$ If $z\in C_{1-\gamma }^{\gamma }(J_{1},E),$ then $z$ satisfies
HNBVP \eqref{e8.1}-\eqref{e8.2} if and only if $z$ satisfies the following
integral equation%
\begin{eqnarray}
z(t) &=&\frac{(t-a)^{\gamma -1}}{\Gamma (\gamma )}\frac{1}{\left(
c+d-A\right) }\sum_{k=1}^{m}\frac{\lambda _{k}}{\Gamma (\mu )}\int_{a}^{\tau
_{k}}(\tau _{k}-s)^{\mu -1}f(s,z(s))ds  \notag \\
&&-\frac{(t-a)^{\gamma -1}}{\Gamma (\gamma )}\frac{d}{\left( c+d-A\right) }%
\frac{1}{\Gamma (1-\gamma +\mu )}\int_{a}^{b}(b-s)^{-\gamma +\mu }f(s,z(s))ds
\notag \\
&&+\frac{1}{\Gamma (\mu )}\int_{a}^{t}(t-s)^{\mu -1}f(s,z(s))ds,  \label{ee3}
\end{eqnarray}%
where $\displaystyle{A=\sum_{k=1}^{m}\lambda _{k}\frac{(\tau _{k}-a)^{\gamma
-1}}{\Gamma (\gamma )}}$, and $c+d\neq A$.
\end{lemma}

Proof: \ In view of Lemma \ref{le}, the solution of \eqref{e8.1} can be
written as%
\begin{equation}
z(t)=\frac{I_{a^{+}}^{1-\gamma }z(a^{+})}{\Gamma (\gamma )}(t-a)^{\gamma -1}+%
\frac{1}{\Gamma (\mu )}\int_{a}^{t}(t-s)^{\mu -1}f(s,z(s))ds,\quad t>a.
\label{e8.3}
\end{equation}

Applying $I_{a^{+}}^{1-\gamma }$ on both sides of \eqref{e8.3} and taking
the limit $t\rightarrow b^{-}$, we obtain
\begin{equation}
I_{a^{+}}^{1-\gamma }z(b^{-})=I_{a^{+}}^{1-\gamma }z(a^{+})+\frac{1}{\Gamma
(1-\gamma +\mu )}\int_{a}^{b}(b-s)^{-\gamma +\mu }f(s,z(s))ds.  \label{e8.4}
\end{equation}%
Now, we substitute $t=\tau _{k}$ in (\ref{e8.3}) and multiply by $\lambda
_{k}$ to obtain%
\begin{equation}
\lambda _{k}z(\tau _{k})=\lambda _{k}\left[ \frac{I_{a^{+}}^{1-\gamma
}z(a^{+})}{\Gamma (\gamma )}(\tau _{k}-a)^{\gamma -1}+\frac{1}{\Gamma (\mu )}%
\int_{a}^{\tau _{k}}(\tau _{k}-s)^{\mu -1}f(s,z(s))ds\right] .  \label{e8.3a}
\end{equation}%
Using the nonlocal boundary condition (\ref{e8.2}) with (\ref{e8.4}) and (%
\ref{e8.3a}), we have%
\begin{eqnarray*}
I_{a^{+}}^{1-\gamma }z(a^{+}) &=&\frac{1}{c}\sum_{k=1}^{m}\lambda _{k}z(\tau
_{k})-\frac{d}{c}I_{a^{+}}^{1-\gamma }z(a^{+}) \\
&&+\frac{d}{c\Gamma (1-\gamma +\mu )}\int_{a}^{b}(b-s)^{-\gamma +\mu
}f(s,z(s))ds.
\end{eqnarray*}%
Therefore, by (\ref{e8.3a}), we have
\begin{eqnarray}
I_{a^{+}}^{1-\gamma }z(a^{+}) &=&\frac{1}{c}\sum_{k=1}^{m}\lambda _{k}\frac{%
I_{a^{+}}^{1-\gamma }z(a^{+})}{\Gamma (\gamma )}(\tau _{k}-a)^{\gamma -1}
\notag \\
&&+\frac{1}{c}\sum_{k=1}^{m}\frac{\lambda _{k}}{\Gamma (\mu )}\int_{a}^{\tau
_{k}}(\tau _{k}-s)^{\mu -1}f(s,z(s))ds  \notag \\
&&-\frac{d}{c}I_{a^{+}}^{1-\gamma }z(a^{+})-\frac{d}{c}\frac{1}{\Gamma
(1-\gamma +\mu )}\int_{a}^{b}(b-s)^{-\gamma +\mu }f(s,z(s))ds.  \notag \\
&=&\frac{1}{\left( c+d-A\right) }\sum_{k=1}^{m}\frac{\lambda _{k}}{\Gamma
(\mu )}\int_{a}^{\tau _{k}}(\tau _{k}-s)^{\mu -1}f(s,z(s))ds  \notag \\
&&-\frac{d}{\left( c+d-A\right) }\frac{1}{\Gamma (1-\gamma +\mu )}%
\int_{a}^{b}(b-s)^{-\gamma +\mu }f(s,z(s))ds,  \label{t1}
\end{eqnarray}%
Submitting \eqref{t1} into \eqref{e8.3}, we obtain%
\begin{eqnarray}
z(t) &=&\frac{(t-a)^{\gamma -1}}{\Gamma (\gamma )}\frac{1}{\left(
c+d-A\right) }\sum_{k=1}^{m}\frac{\lambda _{k}}{\Gamma (\mu )}\int_{a}^{\tau
_{k}}(\tau _{k}-s)^{\mu -1}f(s,z(s))ds  \notag \\
&&-\frac{(t-a)^{\gamma -1}}{\Gamma (\gamma )}\frac{d}{\left( c+d-A\right) }%
\frac{1}{\Gamma (1-\gamma +\mu )}\int_{a}^{b}(b-s)^{-\gamma +\mu }f(s,z(s))ds
\notag \\
&&+\frac{1}{\Gamma (\mu )}\int_{a}^{t}(t-s)^{\mu -1}f(s,z(s))ds.  \label{E5}
\end{eqnarray}

Conversely, applying $I_{a^{+}}^{1-\gamma }$ on both sides of \eqref{ee3},
using Lemma \ref{def8.5} and \ref{Le1}, some simple computations gives
\begin{eqnarray*}
&&I_{a^{+}}^{1-\gamma }\big(cz(a^{+})+dz(b^{-})\big) \\
&=&\frac{c}{\left( c+d-A\right) }\sum_{k=1}^{m}\frac{\lambda _{k}}{\Gamma
(\mu )}\int_{a}^{\tau _{k}}(\tau _{k}-s)^{\mu -1}f(s,z(s))ds \\
&&-\frac{cd}{\left( c+d-A\right) }\frac{1}{\Gamma (1-\gamma +\mu )}%
\int_{a}^{b}(b-s)^{-\gamma +\mu }f(s,z(s))ds \\
&&+\frac{d}{\left( c+d-A\right) }\sum_{k=1}^{m}\frac{\lambda _{k}}{\Gamma
(\mu )}\int_{a}^{\tau _{k}}(\tau _{k}-s)^{\mu -1}f(s,z(s))ds \\
&&-\frac{d^{2}}{\left( c+d-A\right) }\frac{1}{\Gamma (1-\gamma +\mu )}%
\int_{a}^{b}(b-s)^{-\gamma +\mu }f(s,z(s))ds \\
&&+\frac{d}{\Gamma (1-\gamma +\mu )}\int_{a}^{b}(b-s)^{-\gamma +\mu
}f(s,z(s))ds.
\end{eqnarray*}%
\begin{eqnarray*}
&&I_{a^{+}}^{1-\gamma }\big(cz(a^{+})+dz(b^{-})\big) \\
&=&\left( \frac{c}{\left( c+d-A\right) }+\frac{d}{\left( c+d-A\right) }%
\right) \sum_{k=1}^{m}\frac{\lambda _{k}}{\Gamma (\mu )}\int_{a}^{\tau
_{k}}(\tau _{k}-s)^{\mu -1}f(s,z(s))ds \\
&&-\left( d-\frac{cd}{\left( c+d-A\right) }-\frac{d^{2}}{\left( c+d-A\right)
}\right) \int_{a}^{b}\frac{(b-s)^{-\gamma +\mu}}{\Gamma (1-\gamma +\mu )}f(s,z(s))ds \\
&=&\frac{c+d}{\left( c+d-A\right) }\sum_{k=1}^{m}\frac{\lambda _{k}}{\Gamma
(\mu )}\int_{a}^{\tau _{k}}(\tau _{k}-s)^{\mu -1}f(s,z(s))ds \\
&&-\frac{Ad}{\left( c+d-A\right) }\frac{1}{\Gamma (1-\gamma +\mu )}%
\int_{a}^{b}(b-s)^{-\gamma +\mu }f(s,z(s))ds
\end{eqnarray*}%
From (\ref{e8.3a}) and (\ref{t1}), we conclude that%
\begin{equation*}
I_{a^{+}}^{1-\gamma }\big(cz(a^{+})+dz(b^{-})\big)=\sum_{k=1}^{m}\lambda
_{k}z(\tau _{k}),
\end{equation*}%
which shows that the boundary condition (\ref{e8.2}) is satisfied. \newline

Next, applying $D_{a^{+}}^{\gamma }$ on both sides of \eqref{ee3} and using
Lemma \ref{Le1} and \ref{def8.8}, we have
\begin{equation}
D_{a^{+}}^{\gamma }z(t)=D_{a^{+}}^{\nu (1-\mu )}f\big(t,z(t)\big).
\label{e8.9}
\end{equation}

Since $z\in C_{1-\gamma }^{\gamma }(J_{1},E)$ and by definition of $%
C_{1-\gamma }^{\gamma }(J_{1},E)$, we have $D_{a^{+}}^{\gamma }z\in
C_{1-\gamma }(J_{1},E)$, therefore, $D_{a^{+}}^{\nu (1-\mu
)}f=DI_{a^{+}}^{1-\nu (1-\mu )}f\in C_{1-\gamma }(J_{1},E).$ For $f\in
C_{1-\gamma }(J_{1},E)$, it is clear that $I_{a^{+}}^{1-\nu (1-\mu )}f\in
C_{1-\gamma }(J_{1},E)$. Hence $f$ and $I_{a^{+}}^{1-\nu (1-\mu )}f$ satisfy
the hypothesis of Lemma \ref{Le2}.

Now, applying $I_{a^{+}}^{\nu (1-\mu )}$ on both sides of \eqref{e8.9}, we
have%
\begin{equation*}
{\large I_{a^{+}}^{\nu (1-\mu )}}D_{a^{+}}^{\gamma }z(t)={\large %
I_{a^{+}}^{\nu (1-\mu )}}D_{a^{+}}^{\nu (1-\mu )}f\big(t,z(t)\big).
\end{equation*}%
Using Remark~\ref{rem8.a} (i), relation (\ref{e8.9}) and Lemma \ref{Le2}, we
get%
\begin{equation*}
I_{a^{+}}^{\gamma }D_{a^{+}}^{\gamma }z(t)=f\big(t,z(t)\big)-\frac{%
I_{a^{+}}^{1-\nu (1-\mu )}f\big(a,z(a)\big)}{\Gamma (\nu (1-\mu ))}%
(t-a)^{\nu (1-\mu )-1},\, \text{for all}\,\, t\in J_{2}.
\end{equation*}%
\ By Lemma \ref{def8.7}, we have $I_{a^{+}}^{1-\nu (1-\mu )}f\big(a,z(a)\big)%
=0$. Therefore $D_{a^{+}}^{\mu ,\nu }z(t)=f\big(t,z(t)\big)$. This completes
the proof.

To prove the existence of solutions for the problem at hand, let us make the
following hypotheses.

\begin{itemize}
\item[ (H1)] The function $f:J_{2}\times E\rightarrow E$ satisfies the Carath%
\`{e}odory conditions.

\item[ (H2)] $f:J_{2}\times E\rightarrow E$ is a function such that $f(\cdot
,z(\cdot ))\in C_{1-\gamma }^{\nu (1-\mu )}(J_{1},E)$ for any $z\in
C_{1-\gamma }(J_{1},E)$ and there exists $\rho \in L^{p}(J_{1},%
\mathbb{R}
^{+})$ with $p>\frac{1}{\mu }$ and $p>\frac{1}{\gamma }$ such that%
\begin{equation*}
\left\Vert f\big(t,z\big)\right\Vert \leq \rho (t)\left\Vert z(t)\right\Vert %
\big),
\end{equation*}%
for each $t\in J_{2},$ and all $z\in E.$

\item[ (H3)] The inequalities%
\begin{eqnarray*}
\mathcal{G} &:&\mathcal{=}\big(\frac{1}{\Gamma (\gamma )}\frac{\left(
\Lambda _{q,\mu ,\gamma }\right) ^{\frac{1}{q}}}{\left( c+d-A\right) }%
\sum_{k=1}^{m}\frac{\lambda _{k}}{\Gamma (\mu )}(\tau _{k}-a)^{\gamma +\mu
-1}  \notag \\
&&+\big)\frac{1}{\Gamma (\gamma )}\left\vert \frac{d}{\left( c+d-A\right) }%
\right\vert \frac{\left( \Delta _{q,\mu ,\gamma }\right) ^{\frac{1}{q}}}{%
\Gamma (1-\gamma +\mu )}+\frac{\left( \Lambda _{q,\mu ,\gamma }\right) ^{%
\frac{1}{q}}}{\Gamma (\mu )}\big)(b-a)^{\mu }\big)\left\Vert \rho
\right\Vert _{L^{p}}<1,
\end{eqnarray*}%
and%
\begin{eqnarray*}
L^{\ast } &:&=\big(\frac{m}{\Gamma (\gamma )}\frac{(b-a)^{\gamma -1}}{\left(
c+d-A\right) }\sum_{k=1}^{m}\frac{\lambda _{k}(\tau _{k}-a)^{\mu }}{\Gamma
(\mu +1)}  \notag \\
&&+\big(\frac{1}{\Gamma (\gamma )}\left\vert \frac{d}{\left( c+d-A\right) }%
\right\vert \frac{1}{\Gamma (-\gamma +\mu )}+\frac{1}{\Gamma (\mu +1)}\big)%
(b-a)^{\mu }\big)\left\Vert \rho \right\Vert _{L^{p}}<1
\end{eqnarray*}%
hold.
\end{itemize}

Now, we are ready to prove the existence of solutions for the HNBVP %
\eqref{e8.1}-\eqref{e8.2}, which is based on fixed point theorem of M\"{o}%
nch's type.

\begin{theorem}
\label{th8.1} Assume that (H1)-(H3) are satisfied. Then HNBVP \eqref{e8.1}-%
\eqref{e8.2} has at least one solution in $C_{1-\gamma }^{\gamma
}(J_{1},E)\subset C_{1-\gamma }^{\mu ,\nu }(J_{1},E)$.
\end{theorem}

\begin{proof}
Transform the problem \eqref{e8.1}-\eqref{e8.2} into a fixed point problem.
Define the operator ${\large \mathcal{T}}:C_{1-\gamma
}(J_{1},E)\longrightarrow C_{1-\gamma }(J_{1},E)$ as%
\begin{eqnarray}
{\large \mathcal{T}}z(t) &=&\frac{(t-a)^{\gamma -1}}{\Gamma (\gamma )}\frac{1%
}{\left( c+d-A\right) }\sum_{k=1}^{m}\frac{\lambda _{k}}{\Gamma (\mu )}%
\int_{a}^{\tau _{k}}(\tau _{k}-s)^{\mu -1}f(s,z(s))ds  \notag \\
&&-\frac{(t-a)^{\gamma -1}}{\Gamma (\gamma )}\frac{d}{\left( c+d-A\right) }%
\frac{1}{\Gamma (1-\gamma +\mu )}\int_{a}^{b}(b-s)^{-\gamma +\mu }f(s,z(s))ds
\notag \\
&&+\frac{1}{\Gamma (\mu )}\int_{a}^{t}(t-s)^{\mu -1}f(s,z(s))ds.
\label{e8.10}
\end{eqnarray}%
Clearly, from Lemma \ref{lee1}, the fixed points of ${\large \mathcal{T}}$
are solutions to \eqref{e8.1}-\eqref{e8.2}. Let $\mathbb{B}_{R}=\left\{ z\in
C_{1-\gamma }(J_{1},E):\left\Vert z\right\Vert _{C_{1-\gamma }}\leq
R\right\} $. We shall show that $\mathcal{T}$ satisfies the conditions of M%
\"{o}nch's fixed point theorem. The proof will be given in the following
four steps:

Step1: We show that $\mathcal{T}(\mathbb{B}_{R})\subset \mathbb{B}_{R}$. By
definition of $\mathcal{T}$, hypothesis $(H_{2})$ and H\"{o}lder's
inequality, we have%
\begin{eqnarray}
&&\left\Vert (\mathcal{T}z)(t)(t-a)^{1-\gamma }\right\Vert   \notag \\
&=&\frac{1}{\Gamma (\gamma )}\frac{1}{\left( c+d-A\right) }\sum_{k=1}^{m}%
\frac{\lambda _{k}}{\Gamma (\mu )}\int_{a}^{\tau _{k}}(\tau _{k}-s)^{\mu
-1}\left\Vert f(s,z(s))\right\Vert ds  \notag \\
&&+\frac{1}{\Gamma (\gamma )}\left\vert \frac{d}{\left( c+d-A\right) }%
\right\vert \frac{1}{\Gamma (1-\gamma +\mu )}\int_{a}^{b}(b-s)^{-\gamma +\mu
}\left\Vert f(s,z(s))\right\Vert ds  \notag \\
&&+\frac{(t-a)^{1-\gamma }}{\Gamma (\mu )}\int_{a}^{t}(t-s)^{\mu
-1}\left\Vert f(s,z(s))\right\Vert ds  \notag \\
&\leq &\frac{1}{\Gamma (\gamma )}\frac{1}{\left( c+d-A\right) }\sum_{k=1}^{m}%
\frac{\lambda _{k}}{\Gamma (\mu )}\int_{a}^{\tau _{k}}(\tau _{k}-s)^{\mu
-1}(s-a)^{\gamma -1}\rho (s)\left\Vert z\right\Vert _{C_{1-\gamma }}ds
\notag \\
&&+\frac{1}{\Gamma (\gamma )}\left\vert \frac{d}{\left( c+d-A\right) }%
\right\vert \int_{a}^{b}\frac{(b-s)^{-\gamma +\mu}}{\Gamma (1-\gamma +\mu )}
(s-a)^{\gamma -1}\rho (s)\left\Vert z\right\Vert _{C_{1-\gamma }}ds  \notag
\\
&&+\frac{(t-a)^{1-\gamma }}{\Gamma (\mu )}\int_{a}^{t}(t-s)^{\mu
-1}(s-a)^{\gamma -1}\rho (s)\left\Vert z\right\Vert _{C_{1-\gamma }}ds
\notag \\
&\leq &\frac{1}{\Gamma (\gamma )}\sum_{k=1}^{m}%
\frac{\lambda _{k}}{\Gamma (\mu )}\left( \int_{a}^{\tau _{k}}\frac{(\tau_{k}-s)^{(\mu -1)q}}
{\left( c+d-A\right)}(s-a)^{(\gamma -1)q}ds\right) ^{\frac{1}{q}}\left\Vert
\rho \right\Vert _{L^{p}}\left\Vert z\right\Vert _{C_{1-\gamma }}  \notag \\
&&+\frac{1}{\Gamma (\gamma )}\left\vert \frac{d}{\left( c+d-A\right) }%
\right\vert \left(\int_{a}^{b}\frac{(b-s)^{(-\gamma +\mu )q}}{\Gamma (1-\gamma +\mu )}(s-a)^{(\gamma -1)q}ds\right) ^{\frac{1}{%
q}}  \notag \\
&&\times \left\Vert \rho \right\Vert _{L^{p}}\left\Vert z\right\Vert
_{C_{1-\gamma }}+\frac{(t-a)^{1-\gamma }}{\Gamma (\mu )}  \notag \\
&&\times \left( \int_{a}^{t}(t-s)^{(\mu -1)q}(s-a)^{(\gamma -1)q}ds\right) ^{%
\frac{1}{q}}\left\Vert \rho \right\Vert _{L^{p}}\left\Vert z\right\Vert
_{C_{1-\gamma }}.  \label{q1}
\end{eqnarray}%
Since $q>1,$ $p>\frac{1}{\mu }$ and $\frac{1}{p}+\frac{1}{q}=1,$ the change
of variable $s=a-u(\tau _{k}-a)$ yields
\begin{equation}
\left( \int_{a}^{\tau _{k}}(\tau _{k}-s)^{(\mu -1)q}(s-a)^{(\gamma
-1)q}ds\right) ^{\frac{1}{q}}\leq \left( \Lambda _{q,\mu ,\gamma }\right) ^{%
\frac{1}{q}}(\tau _{k}-a)^{\gamma +\mu -1},  \label{e1}
\end{equation}%
the change of variable $s=a-u(b-a)$ gives
\begin{equation}
\left( \int_{a}^{b}(b-s)^{(-\gamma +\mu )q}(s-a)^{(\gamma -1)q}ds\right) ^{%
\frac{1}{q}}\leq \left( \Delta _{q,\mu ,\gamma }\right) ^{\frac{1}{q}%
}(b-a)^{\mu },  \label{e2}
\end{equation}%
and the change of variable $s=a-u(t-a)$ gives us
\begin{equation}
\left( \int_{a}^{t}(t-s)^{(\mu -1)q}(s-a)^{(\gamma -1)q}ds\right) ^{\frac{1}{%
q}}\leq \left( \Lambda _{q,\mu ,\gamma }\right) ^{\frac{1}{q}}(t-a)^{\gamma
+\mu -1},  \label{e3}
\end{equation}%
where
\begin{equation*}
\Lambda _{q,\mu ,\gamma }:=\frac{\Gamma (q(\mu -1)+1)\Gamma (q(\gamma -1)+1)%
}{\Gamma (q(\mu +\gamma -2)+2)},
\end{equation*}%
and
\begin{equation*}
\Delta _{q,\mu ,\gamma }:=\frac{\Gamma (q(\mu -\gamma )+1)\Gamma (q(\gamma
-1)+1)}{\Gamma (q(\mu -1)+2)}.
\end{equation*}

Substitution of (\ref{e1}),(\ref{e2}) and (\ref{e3}) into (\ref{q1}) leads%
\begin{eqnarray*}
&&\left\Vert (\mathcal{T}z)(t)(t-a)^{1-\gamma }\right\Vert \\
&\leq &\frac{1}{\Gamma (\gamma )}\frac{1}{\left( c+d-A\right) }\sum_{k=1}^{m}%
\frac{\lambda _{k}}{\Gamma (\mu )}\left( \Lambda _{q,\mu ,\gamma }\right) ^{%
\frac{1}{q}}(\tau _{k}-a)^{\gamma +\mu -1}\left\Vert \rho \right\Vert
_{L^{p}}\left\Vert z\right\Vert _{C_{1-\gamma }} \\
&&+\frac{1}{\Gamma (\gamma )}\left\vert \frac{d}{\left( c+d-A\right) }%
\right\vert \frac{1}{\Gamma (1-\gamma +\mu )}\left( \Delta _{q,\mu ,\gamma
}\right) ^{\frac{1}{q}}(b-a)^{\mu }\left\Vert \rho \right\Vert
_{L^{p}}\left\Vert z\right\Vert _{C_{1-\gamma }} \\
&&+\frac{(t-a)^{1-\gamma }}{\Gamma (\mu )}\left( \Lambda _{q,\mu ,\gamma
}\right) ^{\frac{1}{q}}(t-a)^{\gamma +\mu -1}\left\Vert \rho \right\Vert
_{L^{p}}\left\Vert z\right\Vert _{C_{1-\gamma }}.
\end{eqnarray*}

For any $z\in \mathbb{B}_{R},$ we obtain
\begin{align*}
{\Vert{\large \mathcal{T}}z\Vert}_{C_{1-\gamma }} &\leq \bigg(\frac{1}{%
\Gamma (\gamma )}\frac{\left(\Lambda _{q,\mu,\gamma }\right)^{\frac{1}{q}}}{%
\left( c+d-A\right)}\sum_{k=1}^{m}\frac{\lambda_{k}}{\Gamma(\mu)}%
(\tau_{k}-s)^{\gamma +\mu-1}\bigg) \\
&+\frac{1}{\Gamma (\gamma )}\left\vert \frac{d}{\left( c+d-A\right)}%
\right\vert\frac{\left(\Delta_{q,\mu,\gamma}\right)^{\frac{1}{q}}} {%
\Gamma(1-\gamma+\mu)}+\frac{\left(\Lambda_{q,\mu,\gamma}\right)^{\frac{1}{q}}%
}{\Gamma(\mu)}\big)(b-a)^{\mu}\big)\left\Vert\rho \right\Vert_{L^{p}}R.
\end{align*}

By (H3), we have $\Vert {\large \mathcal{T}}z\Vert _{C_{1-\gamma }}\leq
\mathcal{G}R\leq R,$ that is, ${\large \mathcal{T(}}\mathbb{B}_{R})\subset
\mathbb{B}_{R}.$

Step 2. We shall prove that ${\large \mathcal{T}}$ is completely continuous.%
\newline
The operator ${\large \mathcal{T}}$ is continuous. Let $\{z_{n}\}_{n\in
\mathbb{N}
}$ is a sequence such that $z_{n}\rightarrow z$ in $\mathbb{B}_{R}$. Then
for each $t\in J_{2},$ we have%
\begin{eqnarray*}
&&\left\Vert \big((\mathcal{T}z_{n})(t)-(\mathcal{T}z)(t)\big)%
(t-a)^{1-\gamma }\right\Vert \\
&=&\frac{1}{\Gamma (\gamma )}\frac{1}{\left( c+d-A\right) }\sum_{k=1}^{m}%
\frac{\lambda _{k}}{\Gamma (\mu )}\int_{a}^{\tau _{k}}(\tau _{k}-s)^{\mu
-1}\left\Vert f(s,z_{n}(s))-f(s,z(s))\right\Vert ds \\
&&+\frac{1}{\Gamma (\gamma )}\left\vert \frac{d}{\left( c+d-A\right) }%
\right\vert \int_{a}^{b}\frac{(b-s)^{-\gamma +\mu}}{\Gamma (1-\gamma +\mu )}\left\Vert f(s,z_{n}(s))-f(s,z(s))\right\Vert ds \\
&&+\frac{(t-a)^{1-\gamma }}{\Gamma (\mu )}\int_{a}^{t}(t-s)^{\mu
-1}\left\Vert f(s,z_{n}(s))-f(s,z(s))\right\Vert ds \\
&\leq &\frac{1}{\Gamma (\gamma )}\frac{1}{\left( c+d-A\right) }\sum_{k=1}^{m}%
\frac{\lambda _{k}}{\Gamma (\mu )}\int_{a}^{\tau _{k}}(\tau _{k}-s)^{\mu
-1}(s-a)^{\gamma -1}ds \\
&&\times \left\Vert f\big(\cdot ,z_{n}(\cdot )\big)-f\big(\cdot ,z(\cdot )%
\big)\right\Vert _{C_{1-\gamma }} \\
&&+\frac{1}{\Gamma (\gamma )}\left\vert \frac{d}{\left( c+d-A\right) }%
\right\vert \frac{1}{\Gamma (1-\gamma +\mu )}\int_{a}^{b}(b-s)^{-\gamma +\mu
}(s-a)^{\gamma -1}ds \\
&&\times \left\Vert f\big(\cdot ,z_{n}(\cdot )\big)-f\big(\cdot ,z(\cdot )%
\big)\right\Vert _{C_{1-\gamma }} \\
&&+\frac{(t-a)^{1-\gamma }}{\Gamma (\mu )}\int_{a}^{t}(t-s)^{\mu
-1}(s-a)^{\gamma -1}ds\left\Vert f\big(\cdot ,z_{n}(\cdot )\big)-f\big(\cdot
,z(\cdot )\big)\right\Vert _{C_{1-\gamma }}
\end{eqnarray*}
\begin{eqnarray*}
&&\left\Vert \big((\mathcal{T}z_{n})(t)-(\mathcal{T}z)(t)\big)%
(t-a)^{1-\gamma }\right\Vert \\
&\leq &\frac{1}{\left( c+d-A\right) }\frac{\mathcal{B}(\gamma ,\mu )}{\Gamma
(\mu )\Gamma (\gamma )}\sum_{k=1}^{m}\frac{\lambda _{k}(\tau _{k}-a)^{\gamma
-1+\mu }}{\Gamma (\mu )}\left\Vert f\big(\cdot ,z_{n}(\cdot )\big)-f\big(%
\cdot ,z(\cdot )\big)\right\Vert _{C_{1-\gamma }} \\
&&+\left\vert \frac{d}{\left( c+d-A\right) }\right\vert \frac{(b-a)^{\mu }}{%
\Gamma (\mu +1)}\left\Vert f\big(\cdot ,z_{n}(\cdot )\big)-f\big(\cdot
,z(\cdot )\big)\right\Vert _{C_{1-\gamma }} \\
&&+\frac{(b-a)^{\mu }}{\Gamma (\mu )}\frac{\mathcal{B}(\gamma ,\mu )}{\Gamma
(\mu )}\left\Vert f\big(\cdot ,z_{n}(\cdot )\big)-f\big(\cdot ,z(\cdot )\big)%
\right\Vert _{C_{1-\gamma }}.
\end{eqnarray*}%
By (H1) and the Lebesgue dominated convergence theorem, we have
\begin{equation*}
\Vert (\mathcal{T}z_{n}-\mathcal{T}z)\Vert _{C_{1-\gamma }}\longrightarrow
0~~as~~n\longrightarrow \infty ,
\end{equation*}%
which means that operator $\mathcal{T}$ is continuous on $\mathbb{B}_{R}$.%
\newline

Step 3. $\mathcal{T}(\mathbb{B}_{R})$ is relatively compact.\newline
From Step 1, we have ${\large \mathcal{T(}}\mathbb{B}_{R})\subset \mathbb{B}%
_{R}.$ It follows that ${\large \mathcal{T(}}\mathbb{B}_{R})$ is uniformly
bounded i.e.\ $\mathcal{T}$ maps $\mathbb{B}_{R}$ into itself. Moreover, we
show that operator $\mathcal{T}$ is equicontinuous on $\mathbb{B}_{R}$.
Indeed, for any $a<t_{1}<t_{2}<b$ and $z\in \mathbb{B}_{R}$, we get%
\begin{eqnarray*}
&&\left\Vert (t_{2}-a)^{1-\gamma }\big({\large \mathcal{T}}z\big)%
(t_{2})-(t_{1}-a)^{1-\gamma }\big({\large \mathcal{T}}z\big)%
(t_{1})\right\Vert \\
&\leq &\dfrac{1}{\Gamma (\mu )}\left\Vert (t_{2}-a)^{1-\gamma
}\int_{a}^{t_{2}}(t_{2}-s)^{\mu -1}f\big(s,z(s)\big)ds\right. \\
&&\left. -(t_{1}-a)^{1-\gamma }\int_{a}^{t_{1}}(t_{1}-s)^{\mu -1}f\big(s,z(s)%
\big)ds\right\Vert \\
&\leq &\dfrac{\Vert f\Vert _{C_{1-\gamma }}}{\Gamma (\mu )}\left\Vert
(t_{2}-a)^{1-\gamma }\int_{a}^{t_{2}}(t_{2}-s)^{\mu -1}(s-a)^{\gamma
-1}ds\right. \\
&&\left. -(t_{1}-a)^{1-\gamma }\int_{a}^{t_{1}}(t_{1}-s)^{\mu
-1}(s-a)^{\gamma -1}ds\right\Vert \\
&\leq &\Vert f\Vert _{C_{1-\gamma }}\frac{\mathcal{B}(\gamma ,\mu )}{\Gamma
(\mu )}\left\Vert (t_{2}-a)^{\mu }-(t_{1}-a)^{\mu }\right\Vert ,
\end{eqnarray*}%
which tends to zero as $t_{2}\rightarrow t_{1},$ independent of $z\in
\mathbb{B}_{R}$, where $\mathcal{B(\cdot },\mathcal{\cdot )}$ is a Beta
function. Thus we conclude that $\mathcal{T}(\mathbb{B}_{R})$ is
equicontinuous on $\mathbb{B}_{r}$ and therefore is relatively compact. As a
consequence of Steps 1 to 3 together with Arzela-Ascoli theorem, we conclude
that $\mathcal{T}:\mathbb{B}_{R}\rightarrow \mathbb{B}_{R}$ is completely
continuous operator.

Step 4: The M\"{o}nch condition is satisfied.\newline
Let $\mathcal{V}$ be a subset of$\ \mathbb{B}_{R}$ such that $\mathcal{%
V\subset }\overline{co}\left( \mathcal{T}(\mathcal{V})\cup \{0\}\right) .$ $%
\mathcal{V}$ is bounded and equicontinuous, and therefore the function $%
t\longrightarrow {\large \alpha }(\mathcal{V(}t\mathcal{)})$ is continuous
on $J_{1}.$ By (H2)-(H3), Lemma 2.6, and the properties of the measure $%
{\large \alpha },$ for each $t\in J_{2}$%
\begin{eqnarray*}
{\large \alpha }(\mathcal{V(}t\mathcal{)}) &\leq &{\large \alpha }(\mathcal{T%
}(\mathcal{V})(t)\cup \{0\})\leq {\large \alpha }(\mathcal{T}(\mathcal{V}%
)(t)) \\
&\leq &\frac{1}{\Gamma (\gamma )}\frac{(t-a)^{\gamma -1}}{\left(
c+d-A\right) }\sum_{k=1}^{m}\frac{\lambda _{k}}{\Gamma (\mu )}\int_{a}^{\tau
_{k}}(\tau _{k}-s)^{\mu -1}\rho (s){\large \alpha }(\mathcal{V}(s))ds \\
&&+\frac{1}{\Gamma (\gamma )}\left\vert \frac{d(t-a)^{\gamma -1}}{\left(
c+d-A\right) }\right\vert \frac{1}{\Gamma (1-\gamma +\mu )}%
\int_{a}^{b}(b-s)^{-\gamma +\mu }\rho (s){\large \alpha }(\mathcal{V}(s))ds
\\
&&+\frac{1}{\Gamma (\mu )}\int_{a}^{t}(t-s)^{\mu -1}\rho (s){\large \alpha }(%
\mathcal{V}(s))ds \\
&\leq &\frac{1}{\Gamma (\gamma )}\frac{(b-a)^{\gamma -1}}{\left(
c+d-A\right) }\sum_{k=1}^{m}\frac{\lambda _{k}}{\Gamma (\mu )}\left(
\int_{a}^{\tau _{k}}(\tau _{k}-s)^{(\mu -1)q}ds\right) ^{\frac{1}{q}%
}\left\Vert \rho \right\Vert _{L^{p}}m{\large \alpha }(\mathcal{V}(b)) \\
&&+\frac{1}{\Gamma (\gamma )}\left\vert \frac{d(b-a)^{\gamma -1}}{\left(
c+d-A\right) }\right\vert \frac{1}{\Gamma (1-\gamma +\mu )}\left(
\int_{a}^{b}(b-s)^{(-\gamma +\mu )q}ds\right) ^{\frac{1}{q}} \\
&&\times \left\Vert \rho \right\Vert _{L^{p}}{\large \alpha }(\mathcal{V}%
(b))+\frac{1}{\Gamma (\mu )}\left( \int_{a}^{t}(t-s)^{(\mu -1)q}ds\right) ^{%
\frac{1}{q}}\left\Vert \rho \right\Vert _{L^{p}}{\large \alpha }(\mathcal{V}%
(b)).
\end{eqnarray*}%
where we have used the fact that
\begin{equation*}
\frac{1}{q}<1\Longrightarrow \frac{1}{(\mu -1)q+1}<\frac{1}{\mu },\text{ }%
0<\mu <1,
\end{equation*}%
and%
\begin{equation*}
\frac{1}{q}<1\Longrightarrow \frac{1}{(-\gamma +\mu )q+1}<\frac{1}{(-\gamma
+\mu )},\text{ }0<\mu <\gamma <1.
\end{equation*}%
Hence%
\begin{eqnarray*}
{\large \alpha }(\mathcal{V}(t)) &\leq &\big(\frac{m}{\Gamma (\gamma )}\frac{%
(b-a)^{\gamma -1}}{\left( c+d-A\right) }\sum_{k=1}^{m}\frac{\lambda
_{k}(\tau _{k}-a)^{\mu }}{\Gamma (\mu +1)} \\
&&+\frac{1}{\Gamma (\gamma )}\left\vert \frac{d}{\left( c+d-A\right) }%
\right\vert \frac{(b-a)^{\mu }}{\Gamma (-\gamma +\mu )}+\frac{(t-a)^{\mu }}{%
\Gamma (\mu +1)}\big)\left\Vert \rho \right\Vert _{L^{p}}<1.
\end{eqnarray*}%
It follows that%
\begin{equation*}
\left\Vert {\large \alpha }(\mathcal{V})\right\Vert _{L^{\infty }}(1-L^{\ast
})\leq 0.
\end{equation*}

This means $\left\Vert {\large \alpha }(\mathcal{V})\right\Vert _{L^{\infty
}}=0,$ i.e. ${\large \alpha }(\mathcal{V}(t))=0$\ for all $t\in J_{2}.$ Thus
$\mathcal{V}(t)$ is relatively compact in $E$. In view of Arzela-Ascoli
theorem, $\mathcal{V}$ is relatively compact in $\mathbb{B}_{R}.$ An
application of Theorem 2.13 shows that $\mathcal{T}$ has a fixed point which
is a solution of HNBVP (\ref{e8.1})-(\ref{e8.2}). The proof is complete.
\end{proof}

\section{An example \label{Sec5}}

We consider the Hilfer fractional differential equation with nonlocal
boundary condition%
\begin{equation}
\begin{cases}
D_{0^{+}}^{\mu ,\nu }z(t)=f\big(t,z(t)\big),\qquad t\in (0,1],0<\mu
<1,0\leq \nu \leq 1, \\
I_{0^{+}}^{1-\gamma }\big[\frac{1}{4}z(0^{+})+\frac{3}{4}z(1^{-})\big]=\frac{%
2}{5}z(\frac{2}{3}),\qquad\mu \leq \gamma =\mu +\nu -\mu \nu ,%
\end{cases}
\label{3}
\end{equation}%
where $ f\big(t,z(t)\big)=\frac{1}{16}t\sin \left\vert z(t)\right\vert,$ $\mu=\frac{1}{3},$ $\nu=\frac{1}{4}$, $\gamma =\frac{1}{2}$, $c=\frac{1}{4}$, $d=\frac{3}{4}$, $\lambda_{1}=\frac{2}{5}$ and $\tau_{1}=\frac{2}{3}.$
Let $E=\mathbb{R}^{+}$ and $J_{2}=(0,1].$

Clearly we can see that $\sqrt{t}f\big(t,z(t)\big)=\frac{1}{16}\sqrt[3]{t}\sin{z(t)}\in C([0,1],
\mathbb{R}^{+}),$ and hence $f\big(t,z(t)\big)\in C_{\frac{1}{2}}([0,1],\mathbb{R}^{+}).$
Also, observe that, for $t\in (0,1]$ and for any $z\in C_{1-\frac{1}{2}}([0,1],\mathbb{R}^{+})$:
\begin{equation*}
\left\Vert f\big(t,z(t)\big)\right\Vert =\left\Vert \frac{1}{16}t\sin
\left\vert z(t)\right\vert \right\Vert \leq \frac{1}{16}t\left\Vert
z(t)\right\Vert .
\end{equation*}
Therefore, the conditions (H1) and (H2) is satisfied with $\rho (t)=\frac{1}{%
16}t.$ Select $p=\frac{1}{2},$ we have $\displaystyle\left\Vert p\right\Vert
_{L^{\frac{1}{2}}}=\left\Vert p\right\Vert _{L^{\frac{1}{2}}}=\left(
\int_{0}^{1}\left\vert \frac{1}{16}s\right\vert ^{\frac{1}{2}}ds\right) ^{2}=%
\frac{1}{48}$. It is easy to check that conditions in (H3) are satisfied
too. Indeed, by some simple computations with $q=\frac{1}{2}$, we get
\begin{equation*}
\Lambda _{q,\mu ,\gamma }:=\frac{\Gamma (q(\mu -1)+1)\Gamma (q(\gamma -1)+1)%
}{\Gamma (q(\mu +\gamma -2)+2)}=\frac{\Gamma (\frac{2}{3})\Gamma (\frac{3}{4}%
)}{\Gamma (\frac{17}{12})},
\end{equation*}%
\begin{equation*}
\Delta _{q,\mu ,\gamma }:=\frac{\Gamma (q(\mu -\gamma )+1)\Gamma (q(\gamma
-1)+1)}{\Gamma (q(\mu -1)+2)}=\frac{\Gamma (\frac{3}{4})\Gamma (\frac{11}{12}%
)}{\Gamma (\frac{5}{3})},
\end{equation*}%
\begin{align*}
\mathcal{G}:=&\bigg(\frac{1}{\Gamma (\gamma )}\frac{\left(
\Lambda _{q,\mu ,\gamma }\right) ^{\frac{1}{q}}}{\left(c+d-A\right) }\frac{
\lambda _{1}}{\Gamma (\mu )}(\tau _{1}-a)^{\gamma +\mu -1}\bigg)+\frac{1}{\Gamma (\gamma )}\left\vert \frac{d}{\left( c+d-A\right)}\right\vert\\
&\times \frac{\left( \Delta_{q,\mu ,\gamma }\right) ^{\frac{1}{q}}}
{\Gamma (1-\gamma +\mu )}+\frac{\left( \Lambda _{q,\mu ,\gamma }\right)^{\frac{1}{q}}}{\Gamma (\mu)}\big((b-a)^{\mu }\big)\left\Vert\rho
\right\Vert _{L^{p}}\simeq 0.03<1
\end{align*}
and
\begin{align*}
L^{\ast}:=&\big(\frac{m}{\Gamma (\gamma )}\frac{(b-a)^{\gamma -1}}{\left(
c+d-A\right) }\sum_{k=1}^{m}\frac{\lambda _{k}(\tau _{k}-s)^{\mu }}{\Gamma
(\mu +1)}+\frac{1}{\Gamma (\gamma )}\left\vert \frac{d}{\left( c+d-A\right)}\right\vert\\
&\times\frac{(b-a)^{\mu }}{\Gamma (-\gamma +\mu )}+\frac{(b-a)^{\mu }}{%
\Gamma (\mu +1)}\big)\left\Vert \rho \right\Vert _{L^{p}}\simeq 0.14<1,
\quad (m=1).
\end{align*}
An application of Theorem \ref{th8.1} implies that problem (\ref{3}) has a
solution in $C_{1-\frac{1}{2}}^{\frac{1}{2}}([0,1],%
\mathbb{R}
^{+})$.


\begin{thebibliography}{99}
\bibitem{ABLZ} Abbas, S., Benchohra, M., Lazreg, J.E. and Zhou, Y. \emph{A
survey on Hadamard and Hilfer fractional differential equations: analysis
and stability}, Chaos Soliton Fract., \textbf{102} (2017), 47--71.

\bibitem{AP1} Abdo, M.S. and Panchal, S.K., \emph{Fractional
integro-differential equations involving $\psi $-Hilfer fractional derivative%
}, Adv. Appl. Math. Mech., \textbf{11} (2019), no. 2,
338--359.

\bibitem{AP2} Abdo, M.S. and Panchal, S.K., \emph{Fractional Boundary value
problem with $\psi -$Caputo fractional derivative}, Proc. Indian Acad. Sci.
(Math. Sci.) (To appear).

\bibitem{APB} Abdo, M.S. and Panchal, S.K., and Bhairat, S.P., \emph{%
Existence of solution for Hilfer fractional differential equations with
boundary value conditions}, (Under review) (2019), 18 pages.

\bibitem{APB} Abdo, M.S. and Panchal, S.K., and Bhairat, S.P., \emph{%
On existence of solution to nonlinear $\psi-$Hilfer Cauchy-type problem}, (Under review) (2019), 17 pages. https://arxiv:submit/2860867[math.CA] 26 Sept 2019.

\bibitem{AG} Agarwal, R.P., Benchohra, M., Hamani, S., \emph{A survey on
existence results for boundary value problems of nonlinear fractional
differential equations and inclusions}, Acta. Appl. Math., \textbf{109}
(2010), 973--1033. DOI 10.1007/s10440-008-9356-6.

\bibitem{SP1} Bhairat, Sandeep P., \emph{Existence and stability results for
fractional differential equations involving generalized Katugampola
derivative}, (In press) Studia U.B.B. Math., 2019, 15 pages. ar.Xiv:1709.08838
[math.CA].

\bibitem{SPN} Bhairat, Sandeep P., and Dhaigude, D.B., \emph{Existence of
solutions of generalized fractional differential equation with with nonlocal
initial condition}, Mathematica Bohemica, \textbf{144} (2019), no. 2,
203--220. DOI: 10.21136/MB.2018.0135-17.

\bibitem{SP5} Bhairat, Sandeep P., \emph{Existence and continuation of solution of
Hilfer fractional differential equations}, J. Math. Modeling, \textbf{7}
(2018), no. 1, 1--20. DOI:10.22124/jmm.2018.9220.1136.

\bibitem{DB1} Bhairat, Sandeep P., \emph{New approach to existence of solution of weighted Cauchy-type problem,}
(Under review) (2019). ar.Xiv.1808.03067[math.CA].

\bibitem{DBN} Bhairat, Sandeep P., and Dhaigude, D.B., \emph{Local existence and
uniqueness of solutions for Hilfer-Hadamard fractional differential problem}%
, Nonlinear Dyn. Syst. Theory, \textbf{18} (2018), no. 2, 144--153.
http://e-ndst.kiev.ua144.

\bibitem{SPB} Bhairat, Sandeep P., \emph{On existence and
approximation of solutions for Hilfer fractional differential equations},
(Under review) arXiv:1704.02464v2 [math.CA], 2017.

\bibitem{DB2} Dhaigude, D.B., and Bhairat, Sandeep P., \emph{Existence and
uniqueness of solution of Cauchy-type problem for Hilfer fractional
differential equations}, Commun. Appl. Anal., \textbf{22}
(2018), no. 1, 121--134. DOI: 10.12732/caa.v22i1.8.

\bibitem{KD} Diethelm, K., \emph{The Analysis of Fractional Differential
Equations}, J Math. Anal. Appl., \textbf{265} (2004), 229--248.
doi:10.1006/jmaa.2000.7194.

\bibitem{FK} Furati, K.M., and Kassim, M.D., \emph{Existence and uniqueness
for a problem involving Hilfer fractional derivative}, Comput. Math.
Applic., \textbf{64} (2012), 1616--1626.

\bibitem{KM} Furati, K.M., and Tatar, N.E., \emph{An existence result for a
nonlocal fractional differential problem}, J. Fractional Cal., \textbf{26}
(2004), 43--54.

\bibitem{FMG} Gaafar, F.M., \emph{Continuous and integrable solutions of a
nonlinear Cauchy problem of fractional order with nonlocal coditions}, J.
Egypt. Math. Soc., \textbf{22} (2014), 341--347.
DOI:http://dx.doi.org/10.1016/j.joems.2013.12.008.

\bibitem{GD} Granas, A., and Dugundji, J., \emph{Fixed Point Theory},
Springer Monographs in Mathematics, Springer-Verlag, New York, 2003.

\bibitem{GC} Gonzalez, C., Melado, A.J., and Fuster, E.L., \emph{A Mönch type fixed point theorem under the interior condition,}
J. Math. Anal. Appl., \textbf{352} (2009), 816-–821. doi:10.1016/j.jmaa.2008.11.032.


\bibitem{HI} Hilfer, R., \emph{Applications of Fractional Calculus in Physics%
}, World Scientific, Singapore, 2000.

\bibitem{HLT} Hilfer, R., Luchko, Y., and Tomovski, Z., \emph{Operational
method for the solution of fractional differential equations with
generalized Riemann-Lioville fractional derivative}, Fract. Calc. Appl.
Anal., \textbf{12} (2009), 289--318

\bibitem{KL1} Kilbas, A.A., Srivastava, H.M., and Trujillo, J.J.,
\emph{Theory and Applications of Fractional Differential Equations}, North-Holland
Math. Stud., 204 Elsevier, Amsterdam 2006.

\bibitem{MH} Monch, H.,
\emph{Boundary value problem for nonlinear ordinary differential equations of second order in Banach spaces,}
Nonlinear Anal., \textbf{75} (1980) no. 5, 985--999.

\bibitem{SA} Sabri, T.M., Ahmed, B., and Agarwal, R.P., \emph{\ On abstract
Hilfer fractional integrodifferential equations with boundary conditions},
Arab J. Math. Sci., (2019), (To appear).

\bibitem{VE} Vivek, D., Kanagarajan, K., and E. M. Elsayed, \emph{Some
existence and stability results for Hilfer-fractional implicit differential
equations with nonlocal conditions}, Mediterr. J. Math.,
\textbf{15}(2018), no. 1, 1--15.

\bibitem{WZ} Wang, J., and Zhang, Y., \emph{Nonlocal initial value problems
for differential equations with Hilfer fractional derivative}, Appl.
Math. Comput., \textbf{266} (2015) 850-859.

\bibitem{SZ} Szufla, S., \emph{On the application of measure of noncompactness to existence theorems,}
Rend. Semin. Mat. U. Pad., \textbf{75} (1986), 1--14.

\end{thebibliography}
\end{document}